\documentclass[11pt]{amsproc}
\usepackage{amsfonts}
\usepackage{amsmath,amstext,amsbsy,amssymb}
\usepackage{color}

\newtheorem{theorem}{Theorem}[section]
\newtheorem{lemma}[theorem]{Lemma}

\newtheorem{corollary}[theorem]{Corollary}

\theoremstyle{definition}

\theoremstyle{remark}
\newtheorem{remark}[theorem]{Remark}

\numberwithin{equation}{section}

\newcommand{\la}{\lambda}
\newcommand{\al}{\alpha}
\newcommand{\ep}{\epsilon}
\begin{document}

\title[Macdonald functions at roots of unity and Newton's identity]
{Modular Macdonald functions and
Generalized Newton's identity}
\author{Tommy Wuxing Cai}
\address{School of Mathematical Sciences,
South China University of Technology, Guangzhou, Guangdong 510640, China
and Department of Mathematics,
MIT, Cambridge, MA 02139-4307, USA}
\email{caiwx@scut.edu.cn}
\author{Naihuan Jing}
\address{
Department of Mathematics,
   North Carolina State University,
   Raleigh, NC 27695-8205, USA}
\email{jing@math.ncsu.edu}
\author{Jian Zhang}
\address{School of Mathematical Sciences,
South China University of Technology, Guangzhou, Guangdong 510640, China}
\email{j.zhang1729@gmail.com}
\thanks{*Corresponding author: Naihuan Jing}
\keywords{Newton's identities, Macdonald functions, modular Hall-Littlewood polynomials, vertex operators}
\subjclass[2010]{Primary: 05E05; Secondary: 17B69, 05E10}

\begin{abstract}
Based on a generalized Newton's identity, we construct a family of symmetric functions which deform the modular
Hall-Littlewood functions.
\end{abstract}
\maketitle
\section{Introduction}

Let $\Lambda_F$ be the ring of symmetric functions over the field $F=\mathbb Q(q, t)$. Macdonald \cite{M}
introduced an orthogonal family of symmetric functions $Q_{\lambda}(q, t)$ and their
dual family $P_{\lambda}(q, t)$ indexed by partitions $\lambda$.
When $q=0$, the functions $Q_{\lambda}(t)=Q_{\lambda}(0, t)$ are
the well-known Hall-Littlewood polynomials, which include the Schur Q-functions
as a special example when $t=-1$. Morris and Saltana \cite{MS} studied the
Hall-Littlewood polynomials $Q_{\lambda}(\xi_m)$ when $t=\xi_m$ an $m$th root of $1$ and
suggested that they have close relation with
modular representation theory of the symmetric group, which was confirmed by
\cite{LLT}. Recently a spin invariant theory was also established by
Wan and Wang \cite{WW1, WW2} where a strong representation theoretic explanation
was found for a similar but different deformation of Schur Q-functions.

It is well-known that this kind of specialization in Hall-Littlewood polynomials to
Schur Q-functions or more generally modular Hall-Littlewood functions
must be dealt separately with
care, because the original
argument for the general case often does not work
at the modular case. For example, in \cite{M} one section was devoted to Schur $Q$-functions indexed by
strict partitions.

In this work we consider the subring $\Lambda^{(m)}$ of symmetric functions generated by
power-sum symmetric functions associated with $m$-regular partitions (i.e. $m$ does not divide
each part). We will directly
prove that there exists a family of orthogonal polynomials $Q_{\lambda}(q, \xi_m)$, $\lambda$
running through $m$-reduced partitions (i.e. with multiplicities $m_i(\lambda)<m$), which deforms the modular Hall-Littlewood polynomials
$Q_{\lambda}(\xi_m)$.

The Macdonald polynomials have been characterized as the eigenfunctions of the Macdonald operator
\cite{M} and the Macdonald symmetric functions were also shown to be eigenfunctions
of certain vertex operator like operators \cite{AMOS, CW, GH, S, NS} (also see \cite{CJ2, CJ4}).
We follow the same strategy together
with a new technique of Newton's identity \cite{CJ4} in the modular cases.
We will construct a graded differential operator $X(z)=\sum_nX_nz^{-n}$ on the space $\Lambda^{(m)}$
with certain triangular property.
The modular Macdonald polynomials
are shown to be the distinguished eigenvectors for the differential operator $X_0$.
We remark that this differential operator is not a specialization of the Macdonald operator
and its limit.
It is an open problem whether there exists two parameter deformation of the modular Hall-Littlewood
polynomials, which are suggested by
the interesting case of Schur Q-functions \cite{WW1, WW2}.

The paper is organized as follows. In section two we first recall the background information of
symmetric functions and give special attention to the subring $\Lambda^{(m)}$ generated by
symmetric functions indexed by $m$-regular partitions. We study the main technique of the
generalized Newton identity for the generalized
complete homogeneous polynomials in section three. We then construct
certain differential operator acting on
the subring $\Lambda^{(m)}$ in section four where some of its basic properties
are studied. In particular we prove a generalized Newton-like identity
and then derive the existence of generalized Hall-Littlewood functions.

\section{Partitions and modular symmetric functions}
First we fix the notations for partitions following \cite{M}.
A partition $\la=(\la_1,\cdots,\la_l)$ of $n$, denoted by $\la\vdash n$, is a sequence of
weakly decreasing non-negative integers such that $n=\sum_i\la_i$, where $\la_i$ are called its parts
and $n$ is called the weight and denoted by
$|\la|=n$. When the parts are arranged increasingly
$\la=(1^{m_1}2^{m_2}\cdots)$, then
$m_i=m_i(\la)$ is called the multiplicity of $i$ in $\la$.
With this notation the length of
$\la$ is defined to be $l(\la)=\sum_im_i$. The
union $\la\cup\mu$ of two partitions $\lambda$ and $\mu$ is defined by
$m_i(\la\cup\mu)=m_i(\la)+m_i(\mu)$ for all $i$. Let $\mathcal P(n)=\{\lambda\vdash n\}$
be the set of partitions of $n$ and $\mathcal P=\cup_{n=1}^{\infty}\mathcal P(n)$ the set of
all partitions. For $\la,\mu\in\mathcal P(n)$, the dominance order
$\la\geq\mu$ is defined by $\sum_{j=1}^i\lambda_j\geq \sum_{j=1}^i\mu_j$
for
all $i$. We write $\la>\mu$ if $\la\geq\mu$ but $\la\neq\mu$. If
$m_i(\mu)\leq m_i(\la)$ for all $i$, we denote $\mu\subset'\la$ and
define $\la\backslash\mu\in\mathcal P$ by
$m_i(\la\backslash\mu)=m_i(\la)-m_i(\mu)$. We also define $m(\la)!=\prod_{i\geq1}m_i(\la)!$.

 From now on we fix a positive integer $m\geq2$, and let $\xi_m$ be a fixed primitive $m$th root of unity.

If each nonzero part of $\la$
is not divisible by $m$, the partition $\lambda$ is called {\it $m$-regular}
and we denote $m\nmid\la$, and let $\mathcal P^{(m)}$ denote the set
of $m$-regular partitions.  Accordingly $\mathcal P^{(m)}(n)=\mathcal P^{(m)}\cap \mathcal P(n)$.
 If each multiplicity of $\la$ satisfies $m_i(\la)< m$, the partition $\lambda$ is called
 {\it $m$-reduced}, and 
we use $\mathcal P_m$ to denote the set
of $m$-reduced partitions. Similarly $\mathcal P_m(n)=\mathcal P_m\cap \mathcal P(n)$.

 \begin{lemma}\label{L:2typeofpartions}
 For positive integers $m,n$, the set $\mathcal P^{(m)}(n)$ of $m$-regular partitions of weight $n$
 and the set $\mathcal P_m(n)$ of $m$-reduced partitions of weight $n$
  have the same cardinality.
 \end{lemma}
The ring $\Lambda$ of symmetric functions in the variables $x_1, x_2, x_3, \ldots,$ is a $\mathbb Z$-module
with basis $m_{\la}$, $\la\in\mathcal{P}$, where for a partition $\la=(\la_1,\cdots,\la_s)$, the monomial function $m_\la=m_\la(x)=\sum_{\alpha}x_1^{\alpha_1}x_2^{\alpha_2}\cdots$, where $\alpha$ runs over
distinct compositions of $\lambda=(\alpha_1, \cdots, \alpha_l, 0, \cdots)$. When the partition has only one part, i.e. $\la=(n)$, we denote
 $p_n=m_{(n)}=x_1^n+x_2^n+\cdots$ and call it the $n$th power sum symmetric function. For convenience,
we define $p_0=1$ and $p_n=0$ when $n<0$. The power sum symmetric functions $p_\la=p_{\la_1}p_{\la_2}\cdots$
form another basis of $\Lambda_F=\Lambda\otimes_\mathbb{Z}F$ for a field $F$ of characteristic $0$. The vector space $\Lambda_F$ is also a commutative algebra over $F$ with the natural multiplication defined by $p_\la p_\mu=p_{\la\cup\mu}$. It is a graded algebra with degree given by $\deg(p_i)=i$, $i\geq0$.

\section{A generalized Newton's formula and its refinement}

 Let $\mathcal{A}$ be a commutative $F$-algebra with the unit $1$.
Let $\{q_n\}_{n\geq 0}$ be a sequence in $\mathcal{A}$ with $q_0=1$, we
 define $q_{\lambda}=q_{\la_1}q_{\la_2}\dotsm q_{\la_s}$
 for any partition $\la=(\la_1,\dotsc,\la_s)$. For convenience, we set $q_n=0$ for $n<0$.

 The following theorem generalizes the Newton identity (see Remark 3.3 of Corollary 3.2 in \cite{CJ4}).

\begin{theorem}\cite{CJ4}\label{T:Newtongeneralization}
 Assume $\{R_n\}_{n\geq 0}$ is another sequence of $\mathcal A$
 such that $\sum_{i\geq1}R_iq_{n-i}=d_nq_n$ with $d_n\in F$ for each $n\geq1$ and $R_0=1$.
 Then for any partition $\la=(\la_1,\dotsc,\la_s)$ of length $s$, we have
 \begin{equation}\label{F:Traisesq}
 \sum_{i_1,\dotsc,i_s\geq 1}R_{i_1+\dotsm+i_s}q_{\la_1-i_1}\cdots q_{\la_s-i_s}=\sum_{\mu\geq\la}d_{\la\mu}q_\mu,
 \end{equation}
where $d_{\la\mu}\in F$. In particular $d_{\la\la}=(-1)^{s-1}d_{\la_s}$.
\end{theorem}

We can give a combinatorial description of the coefficient $d_{\la\mu}$.
\begin{lemma}
For two partitions $\la=(\la_1,\dotsc,\la_s)$ of length $s$
and $\nu=(\nu_1,\dotsc,\nu_t)$ of length $t$, let $N_l(\la,\nu)$ be the cardinality
of the set
\begin{equation}\label{E:i1i2is}
\{(i_1,\dotsc,i_s)|i_1,\dotsc,i_s\geq1 \text{ and }
q_{\la_1-i_1}\dotsm q_{\la_s-i_s}=q_\nu\}.
\end{equation}
Then one has that
\begin{equation}
m(\nu)!N_l(\la,\nu)=k_1(k_2-1)\dotsm (k_t-(t-1)),
\end{equation} where $k_i$ is the number of $j$'s such that $\la_j>\nu_i$.
Moreover $N_l(\la,\nu)$ is also given by the following formula:
\begin{align}\label{F:Nllamu}
m(\nu)!N_l(\la,\nu)=\prod_{\begin{subarray}{r}{(i,k):i\geq1,}\\ {1\leq k\leq m_i(\nu)}\end{subarray}}\Big(1-k+\sum_{j\geq
i+1}\big(m_j(\la)-m_j(\nu)\big)\Big).
\end{align}
\end{lemma}
\begin{proof}
For the first statement, we only need to consider the case that $s\geq t$. Otherwise $N_l(\la,\nu)=0$ and it is easily seen that the statement is true.
For convenience, we write $\nu$ in the following form:
\begin{equation}
\nu=(a_1a_1\dotsm a_1 a_2 a_2 \dotsm a_2\dotsm )=(a_1^{m_1}a_2^{m_2}\dotsm a_r^{m_r}),
\end{equation}
where $a_1>a_2>\dotsm>a_{r-1}>a_r\geq0$ and $m_1+\dotsm+m_r=s$. Thus if $s>t$, then $a_r=0$ and $m_r=s-t$.
Now let $S$ be the set of partitions $\mathcal{M}=(M_1,\dotsc,M_r)$ of the set $\{1,\dotsc,s\}$ such that for $d=1,\dotsc,r$
\begin{enumerate}
  \item the cardinality of $M_d$ is $m_d$; i.e. $|M_d|=m_d$,
  \item $j\in M_d\Rightarrow \la_j>a_d.$
\end{enumerate}

We want to define a bijection $\phi$ from $S$ to the set (\ref{E:i1i2is}). For $\mathcal{M}=(M_1,\dotsc, M_r)\in S$, we define a $\phi(\mathcal{M})=(i_1,\dotsc,i_s)$ by letting $i_j=\la_j-a_d$ for $j\in M_d$ . By definition, $\la_j-i_j=a_d$ if and only if $j\in M_d$ (as $a_1,\dotsc,a_r$ are distinct). This is equivalent to the equality $q_{\la_1-i_1}\dotsm q_{\la_s-i_s}=q_\nu$ (as $\nu=(a_1^{m_1}a_2^{m_2}\dotsm a_r^{m_r})$ and $|M_d|=m_d$). We thus see that $\phi$ is really a bijection.

Hence we only need to find the cardinality of $S$. We consider the set $\widetilde{S}$ of permutations of
$\al=(j_1,j_2,\dotsc,j_{m_1},j_{m_1+1},j_{m_1+2},\dotsc,j_{m_1+m_2},\dotsc)$ of $\{1,\dotsc,s\}$ such that $\la_{j_d}>a_1$ for $d=1,\dotsc,m_1$ and $\la_{j_d}>a_2$ for $d=m_1+1,\dotsc,m_1+m_2$ and so on. We see that $|\tilde{S}|=m_1!\dots m_r!|S|$ and thus
\begin{equation}\label{E:NlanuWidetildeS}
N_l(\la,\nu)=|S|=\frac{|\widetilde{S}|}{m_1!\dotsm m_r!}.
\end{equation}

Now we compute the cardinality of $\widetilde{S}$. To obtain a $\al=(j_1,j_2,\dotsc)\in \widetilde{S}$, we first have $k_1$ choices for $j_1$. After $j_1$ is taken, we have $k_2-1$ choices for $j_2$ and so on. We see that $$|\widetilde{S}|=k_1(k_2-1)\dotsm (k_t-(t-1))(k_{t+1}-t)\dotsm (k_s-(s-1)).$$
Combining with expression (\ref{E:NlanuWidetildeS}), we get formula (\ref{F:Nllamu}) in the case that $s=t$. If $s>t$, then $m_r=s-t$, $m(\nu)!=m_1!\dotsm m_{r-1}!$ and $k_{t+1}=\dotsm=k_s=s$ (as $\nu_{t+1}=\dotsm=\nu_s=0$). So $(k_{t+1}-t)\dotsm (k_s-(s-1))=m_r!$ and we also obtain formula (\ref{F:Nllamu}). We finish proving the first statement.

The second expression comes from the first one. For each $i\geq1$, there is a string of $m_i(\nu)$ $i$'s in $\nu$. For the $k$th $i$ in this string, we denote its position in $\nu$ by $g_\nu(i,k)$ (so in particular $\nu_{g_\nu(i,k)}=i$). We can rewrite the right side of (\ref{F:Nllamu}) in the following form:
\begin{equation}\label{E:expression1}
\prod_{\begin{subarray}{r}{\quad i\geq1;}\\ {1\leq k\leq m_i(\nu)}\end{subarray}} \big(f_{\la}(i)-(g_{\nu}(i,k)-1)\big),
\end{equation}
where $f_{\la}(i)$ is the number of $j$'s such that $\la_j>i$. We see that $f_{\la}(i)=\sum_{j>i}m_j(\la)$, and also $g_{\nu}(i,k)=k+\sum_{j>i}m_j(\nu)$. Plugging these into (\ref{E:expression1}), we finish the proof of (\ref{F:Nllamu}).
\end{proof}

\begin{lemma}\label{L:expansionRn} Under the same hypothesis
of Theorem \ref{T:Newtongeneralization}
we have $R_n=\sum_{\mu\vdash n}d_\mu q_\mu$
where
\begin{equation}\label{F:dmu}
d_\mu=(-1)^{l(\mu)-1}\frac{(l(\mu)-1)!}{m(\mu)!}\sum_k
m_{k}(\mu)d_k.
\end{equation}
\end{lemma}

\begin{proof}
By assumption $q_n+R_1q_{n-1}+R_2q_{n-2}+\cdots+R_{n-1}q_1+R_n=d_nq_n$,
so $R_n=d_nq_n-R_1q_{n-1}-R_2q_{n-2}-\cdots-R_{n-1}q_1$.
Plugging $R_i$ $(i=n-1, \cdots, 1)$ into this equation iteratively, we can write $R_n=\sum_{\mu\vdash
n}d_\mu q_\mu$ with $d_\mu\in F$.

We now show by induction that
$d_\mu$ are given by (\ref{F:dmu}). First when
$\mu$ has only one part, i.e. $\mu=(|\mu|)$, then
$d_\mu=d_{|\mu|}$ so (\ref{F:dmu}) is true. Let $\mu=(1^{m_1}2^{m_2}\dotsm)$ have at least two parts. For
each $m_i\geq1$ set $\mu^i=(1^{m_1}2^{m_2}\cdots
i^{m_i-1}\dotsm)$. Then $d_\mu$ is
$\sum_i -d_{\mu^i}$, where the sum runs over all $i$ with $m_i\geq1$.
Note that $l(\mu^i)=l(\mu)-1$, by induction hypothesis we have
\begin{equation}\label{F:dmui}
d_{\mu^i}=(-1)^{l(\mu)-2}\frac{(l(\mu)-2)!}{m(\mu)!}m_i\big(-d_i+\sum_k
m_{k}(\mu)d_k\big).
\end{equation}
Note that if $m_i=0$, the right side of (\ref{F:dmui}) is zero. Therefore
we have
\begin{align*}
d_\mu&=\sum_i -d_{\mu^i}\\
     &=-(-1)^{l(\mu)-2}\frac{(l(\mu)-2)!}{m(\mu)!}\sum_i m_i\big(-d_i+\sum_k m_k d_k\big)\\
     &=(-1)^{l(\mu)-1}\frac{(l(\mu)-2)!}{m(\mu)!}\big(-\sum_i m_id_i+l(\mu)\sum_k m_k d_k\big)\\
     &=(-1)^{l(\mu)-1}\frac{(l(\mu)-1)!}{m(\mu)!}\sum_k m_k d_k.
\end{align*}
\end{proof}

Now the numbers $d_{\la\mu}$ can be computed as follows.
\begin{corollary}\label{C:dlambdamu}
 The coefficient $d_{\la\mu}$ in Theorem \ref{T:Newtongeneralization} can be expressed as
 \begin{align}
 d_{\la\mu}=\sum_{\nu\subset'\mu}N_l(\la,\nu) d_{\mu\backslash\nu},
 \end{align}
 where $N_l(\la,\nu)$ and $d_\mu$ are defined by Equations (\ref{F:Nllamu}) and (\ref{F:dmu}) respectively.
\end{corollary}

\section{A generalized modular Hall-Littlewood function}

Let $\ep=(\ep_1, \ep_2, \dotsc)$ be a sequence of (finite or infinite) non-zero parameters.
Let $\mathbb{Q}(\ep)=\mathbb{Q}(\epsilon_1,\epsilon_2,\cdots)$ be the field of rational functions in $\ep_1,\ep_2,\dotsc$ over $\mathbb{Q}$.
 The ring $\Lambda_{\mathbb Q}=\Lambda\otimes \mathbb Q$ is a free commutative algebra generated by the
 power-sum symmetric polynomials $p_1,p_2,\dotsc$ over $\mathbb{Q}$, thus the power sums symmetric functions $p_\la=p_{\la_1}p_{\la_2}\dotsm$, $\la\in\mathcal{P}$, form a $\mathbb Q$-linear basis.
By abuse of notation, we let $\Lambda(\ep)$ denote the ring
$\Lambda\otimes\mathbb{Q}(\ep)$. We equip $\Lambda(\ep)$
with the following scalar product \cite{Ke}:
\begin{align} \label{def}
\langle p_{\la}, p_{\mu}\rangle=\delta_{\la \mu}\epsilon_\la
z_\lambda,
\end{align}
where $\delta$ is the Kronecker symbol, $\epsilon_\la=\epsilon_{\la_1}\epsilon_{\la_2}\ldots$ with $\epsilon_0=1$, and $$z_\la=\prod_{i\geq1}i^{m_i(\la)}m_i(\la)!.$$

We are interested in the subalgebra $\Lambda^{(m)}(\epsilon)$ generated by $p_n$'s with $m\nmid n$,
 i.e. $\Lambda^{(m)}(\epsilon)$
 has a linear basis spanned by $p_{\lambda}$, $\lambda\in\mathcal P^{(m)}$.
 This subalgebra is a graded algebra with the natural degree gradation:
\begin{equation}
\Lambda^{(m)}(\epsilon)=\bigoplus_{n=0}^{\infty} \Lambda^{(m)}_n(\epsilon),
\end{equation}
where $\Lambda^{(m)}_n(\epsilon)$ is the subspace of homogeneous symmetric functions
of degree $n$.

For each $\la\in\mathcal{P}$, we define the generalized complete symmetric functions $q_\la=q_{\la_1} q_{\la_2}\cdots$, for which $q_n$ are given by:
\begin{equation}\label{E:generating}
Y_m(z)=\exp\Big(\sum_{m\nmid n\geq1}
\frac{z^n}{n\epsilon_n}p_n\Big)=\sum_{n=0}^{\infty} q_{n} z^n.
\end{equation}
Thus $q_\la$ is homogeneous of degree $|\la|$,
$q_0=1$ and for $n>0$ one has
\begin{equation}
q_{n}=\sum_{\la\in\mathcal P^{(m)}(n)}\frac{p_\la}{z_\la\ep_\la}\in\Lambda^{(m)}_n(\epsilon).
\end{equation}

The following elementary result generalizes well-known bases in $\Lambda$.
\begin{lemma}\label{L:3bases}
The following three sets are all bases of $\Lambda^{(m)}_n(\epsilon)$:
\begin{align}\label{B:p^m}
 &\{p_\la : \la\in\mathcal P^{(m)}(n)\}\\ \label{B:q^m}
&\{q_\la : \la\in\mathcal P^{(m)}(n)\}\\ \label{B:q_m}
 &\{q_\la : \la\in \mathcal P_m(n)\}. 
 \end{align}
\end{lemma}

\begin{proof}
First of all (\ref{B:p^m}) defines a basis by definition.
Applying logarithm to both sides of Eq. (\ref{E:generating}), we find that
\begin{equation}\label{E:q2p}
p_n=\sum_{\la\vdash n}n\ep_n(-1)^{l(\la)-1}\frac{(l(\la)-1)!}{m(\la)!}q_\la,
\end{equation}
for each $n$ such that $m\nmid n$. This implies that each $p_\la$ in (\ref{B:p^m}) is a linear combination of $q_\mu$'s (with $|\mu|=|\la|$).
Notice that the three sets have the same cardinality by Lemma \ref{L:2typeofpartions}.
Then the lemma is a consequence of the following facts: for $k\geq 1$, \\
(a) $q_{km}$ is a linear combination of $q_\mu$'s with $\mu<(km)$,\\
(b) $q_{(k^m)}$ is a linear combination of $q_\nu$'s with $m_i(\nu)<m$ and $\nu>(k^m)$.\\
In fact one can iteratively use (a) ((b) respectively) to express $q_\la$ --and thus $p_\la$-- as a linear combination of the elements in (\ref{B:q^m})((\ref{B:q_m}) respectively).

 Now let us turn to the proof of (a) and (b). Recall that $\xi_m$ is an $m$th primitive root of $1$.
 One has
\begin{align}\label{E:Yproperty}
\prod_{i=1}^m\Big(\sum_{n}q_n\xi_m^{in}z^n\Big)&=\prod_{i=1}^{m}Y_m(\xi_m^i z)\\ \nonumber
&=\exp\Big(\sum_{m\nmid n}\frac{(\xi_m^{n}+\xi_m^{2n}\dotsm+\xi_m^{mn})z^n}{n\epsilon_n}p_n\Big)\\ \nonumber
&=1,
\end{align}
where we use the fact that $\sum_{i=1}^m\xi_m^{in}=0$ if $m\nmid n$.
For every positive integer $k$, considering the coefficient of $z^{km}$ in Equation (\ref{E:Yproperty}),
we have
\begin{equation}
mq_{km}+\sum_{\nu}c_\nu q_\nu+(-1)^{(m+1)k}q_{(k^m)}=0,
\end{equation}
where the sum is over $\nu$ such that $1<l(\nu)\leq m$ and $\nu\neq(k^m)$. Note that among all partitions
$\mu$ such that $|\mu|=km$ and $l(\mu)\leq m$, $(km)$ is the largest and $(k^m)$ is the smallest in terms of the dominance order. Thus $q_{km}$ is a linear combination of some $q_\mu$'s with $\mu<(km)$, and $q_{(k^m)}$ is a linear combination of $q_\mu$'s with $m_i(\mu)<m$ and $\mu>(k^m)$. This finishes the proof.
\end{proof}

\begin{remark}\label{R:qasq'}
From the proof we see that $q_\la$ is a linear combination of $q_\mu$'s with $m\nmid\mu\leq\la$, and it is also a linear combination of $q_\mu$'s with $\mu\geq\la$ and $m_i(\mu)<m$.
\end{remark}

We need some linear operators on $\Lambda^{(m)}(\ep)$. For a positive integer $n$ with $m\nmid n$,
 define $h_n, h_{-n}\in \text{End}_{\mathbb{Q}(\ep)}(\Lambda^{(m)}(\ep))$ by:
\begin{align}
h_n\cdot v&=n\epsilon_n\frac{\partial}{\partial p_{n}}v,\\
h_{-n}\cdot v&=p_{n}v,
\end{align} where $v\in \Lambda^{(m)}(\ep)$.

For an operator $A$ on $\Lambda^{(m)}(\ep)$, the conjugate $A^*$ is defined by $\langle A\cdot u,v\rangle=\langle u,A^*\cdot v\rangle$ ( $u,v\in \Lambda^{(m)}(\ep)$).
 We say $A$ is {\it self-adjoint} if $A=A^*$. We say that $A$ is a {\it raising operator}
for the set $\{q_\la\}$ if $A\cdot q_\la$ is of the form $A\cdot q_\la=\sum_{\mu\geq\la}a_{\la\mu}q_\mu$ for all $q_\la$ in this set.

We can easily prove that $h_n^*=h_{-n}$ and $h_i.q_n=q_{n-i}$ for $n\in\mathbb{Z}$ and $m\nmid i>0$.

 \subsection{A generalization of modular Hall-Littlewood functions}

For the algebra $\Lambda^{(m)}(\epsilon)$ (of symmetric functions), we consider
$\ep_n=\frac{q^n-1}{(1-\xi_m^n)c^n}$, where $m\nmid n$ and $q,c$ are two parameters. We define the following vertex operator
 \begin{align}
 X(z)&=\exp\Big(\sum_{m\nmid n\geq1}
\frac{z^{n}h_{-n}}{n}(1-\xi_m^n)c^n\Big)\exp\Big(\sum_{m\nmid n\geq1}
\frac{z^{-n}h_n}{n}(1-\xi_m^n)\Big)\\ \nonumber
&=\sum_n X_{n} z^{-n},
\end{align}
which maps $\Lambda^{(m)}(\ep)$ to $\Lambda^{(m)}(\ep)[[z]]$, the space of formal power series of $\Lambda^{(m)}(\ep)$
in $z$.
We will show that the eigenvectors of $X_0$ form an orthogonal basis of $\Lambda^{(m)}(\ep)$.

\begin{lemma}\label{L:creatingpart}
Define $R_n\in\Lambda^{(m)}(\epsilon)$ by  $$\exp\Big(\sum_{m\nmid n\geq1}
\frac{z^{n}p_{n}}{n}(1-\xi_m^n)c^n\Big)=\sum_{n\geq0}R_nz^n,$$ then we have $$\sum_{i\geq1}R_iq_{n-i}=(q^n-1)q_n.$$
\end{lemma}
\begin{proof} This follows from the following computation:
\begin{align*}
&\sum_{n\geq0}(\sum_{i\geq0}R_iq_{n-i})w^n=\sum_{i\geq0}R_iw^i\cdot\sum_{j\geq0}q_jw^j\\
&=\exp\Big(\sum_{m\nmid n\geq1}
\frac{w^{n}p_{n}}{n}(1-\xi_m^n)c^n\Big)\exp\Big(\sum_{m\nmid n\geq1}
\frac{w^n p_{n}}{n\epsilon_n}\Big)=\exp\Big(\sum_{m\nmid n\geq1}
\frac{w^np_{n}}{n\epsilon_n}q^n\Big)\\
&=\sum_{n\geq0}q_n(qw)^n.   
\end{align*}
\end{proof}

\begin{theorem}\label{T:target operator}
The operator $X_0$ is self-adjoint and acts on ${q_\la}'s$ as a raising
operator, i.e. $X_0\cdot q_\la=\sum_{\mu\geq\la}c_{\la\mu}q_\mu$.
Moreover the leading coefficient is given by
$$c_{\la\la}=1+(1-\xi_m)\sum_{i=1}^{l(\la)}(q^{\la_i}-1)\xi_m^{i-1}.$$
\end{theorem}
\begin{remark} \label{r:eigenvalue} Combining this with Remark \ref{R:qasq'}, we can also write
$X_0\cdot q_\la=\sum_{\la\leq\mu\in\mathcal P_m}c'_{\la\mu}q_\mu$.
Moreover $c'_{\la\la}=c_{\la\la}$ if $\la\in \mathcal P_m$. 
\end{remark}
\begin{proof}
Let us first define an operator $Y_n$ on $\Lambda(\ep)$ by:
\begin{equation}\label{E:qgenerating}
Y(w)=\exp\Big(\sum_{m\nmid n\geq1}
\frac{w^n}{n\epsilon_n}h_{-n}\Big)=\sum_n Y_{n} w^n.
\end{equation}
Note that the action of $Y_n$ on $\Lambda(\ep)$ is the multiplication by $q_n$, i.e. $Y_{n}\cdot v=q_{n}v$.

For a partition $\la=(\la_1,\dots,\la_s)$,
$X_0\cdot q_\la$ is the coefficient of $z^0w_1^{\la_1}\cdots w_s^{\la_s}$ in the following expression
\begin{align*}
&X(z)Y(w_1)\cdots Y(w_s)\cdot 1\\
&=\exp\Big(\sum_{m\nmid n\geq1}
\frac{z^{n}h_{-n}}{n}(1-\xi_m^n)c^n\Big)\exp\Big(\sum_{m\nmid n\geq1}
\frac{z^{-n}h_n}{n}(1-\xi_m^n)\Big)\\
&\qquad\qquad\cdot\prod_{i=1}^s \exp\Big(\sum_{m\nmid n\geq1}
\frac{w_i^n}{n\epsilon_n}h_{-n}\Big)\cdot 1\\
&=\sum_{n\geq0}R_nz^n\cdot\prod_{i=1}^s \exp\Big(\sum_{m\nmid n\geq1}
\frac{w_i^n}{n\epsilon_n}h_{-n}\Big)\prod_{i=1}^s \exp\Big(\sum_{m\nmid n\geq1}
\frac{(w_i/z)^{n}}{n}(1-\xi_m^n)\Big)\cdot 1\\
&=\sum_{n\geq0}R_nz^n\cdot\sum_{n_1,\dots,n_s\geq0}q_{n_1}w_1^{n_1}\cdots q_{n_s}w_s^{n_s}\cdot\prod_{i=1}^s\frac{1-\xi_m w_i/z}{1-w_i/z}\cdot
\end{align*}
Note that $\frac{1-\xi_mw_i/z}{1-w_i/z}=\sum_{k\geq0}a_k (w_i/z)^k$, where $a_0=1$ and $a_k=1-\xi_m$ for $k\geq1$. Therefore we have that
\begin{align}\nonumber
X_0\cdot q_\la&=\sum_{i_1,\dots,i_s\geq0}R_{i_1+\cdots+i_s}a_{i_1}q_{\la_1-i_1}\cdots a_{i_s}q_{\la_s-i_s}\\ \label{E:actionofX0}
&=\sum_{\lambda_{\underline{j}}}\frac{(1-\xi_m)^{k}q_{\lambda}}{q_{\lambda_{j_1}}\cdots q_{\lambda_{j_k}}}\sum_{i_1,\dots,i_k\geq1}R_{i_1+\cdots+i_k}q_{\la_{j_1}-i_1}\cdots q_{\la_{j_k}-i_k},
\end{align}
where we have pulled
out $q_{\lambda_j}$ whenever $i_j=0$ and $\lambda_{\underline{j}}=(\lambda_{j_1},\ldots ,\lambda_{j_k})$ runs through any subsequence of $\lambda$. By Lemma \ref{L:creatingpart}, $R_n$ and $q_n$ satisfy the
assumption of Theorem \ref{T:Newtongeneralization} with $d_n=q^n-1$. Using Theorem \ref{T:Newtongeneralization}
we can compute the inside sum on the right of
(\ref{E:actionofX0}) as $\sum_{\mu\geq \lambda_{\underline{j}}}d_{\lambda_{\underline{j}}\mu}q_{\mu}$ with the leading coefficient
$(q^{\lambda_{j_k}}-1)(-1)^{k-1}$. Therefore
\begin{align*}
X_0\cdot q_\la&=\sum_{\lambda_{\underline{j}}}(1-\xi_m)^{k}q_{\lambda\backslash\lambda_{\underline{j}}}\sum_{\mu\geq \lambda_{\underline{j}}}d_{\lambda_{\underline{j}}\mu}q_{\mu}\\
&=\sum_{\lambda_{\underline{j}}}\sum_{\mu\geq \lambda_{\underline{j}}}(1-\xi_m)^{k} d_{\lambda_{\underline{j}}\mu}q_{\mu\cup (\lambda\backslash\lambda_{\underline{j}})}\\
&=\sum_{\mu\geq\lambda}c_{\lambda\mu}q_{\mu}.
\end{align*}
The leading coefficient is
\begin{align*}
&1+\sum_{\lambda_{\underline{j}}\neq \emptyset}(1-\xi_m)^{k} (q^{\lambda_{j_k}}-1)(-1)^{k-1}\\
=&1+\sum_{k=1}^{\l(\lambda)}(1-\xi_m)^{k}\sum_{l(\lambda_{\underline{j}})=k} (q^{\lambda_{j_k}}-1)(-1)^{k-1}
\end{align*}
which is equal to $c_{\lambda\lambda}$ by the binomial formula.
\end{proof}

Note that it is possible that $c_{\la\la}=c_{\mu\mu}$ for $\la\neq\mu$. For example, this happens when $\la=(1^k2^{m+l})$, $\mu=(1^{k+2m}2^l)$. However if we restrict the partitions to those with multiplicities less than $m$, than $\la\neq\mu$ does imply $c_{\la\la}\neq c_{\mu\mu}$. To see this, for $\la=(1^{m_1(\lambda)}2^{m_2(\lambda)}\cdots)$ we can rewrite the \emph{main} part of $c_{\la\la}$:
\begin{align*}
f_\la(q)&=\sum_{i=1}^{l(\la)}(q^{\la_i}-1)\xi_m^{i-1}\\
&=(q-1)\sum_{m_1(\lambda)\geq j\geq 1}\xi_m^{l(\la)-j}+(q^2-1)\sum_{m_2(\lambda)\geq j\geq 1}\xi_m^{l(\la)-m_1(\lambda)-j}\\
&\qquad\qquad+(q^3-1)\sum_{m_3\geq j\geq 1}\xi_m^{l(\la)-m_1(\lambda)-m_2(\lambda)-j}+\cdots.
\end{align*}

If $c_{\la\la}=c_{\mu\mu}$ then $f_\la(q)=f_\mu(q)$.  Then we must have $\sum_{i=1}^{l(\la)}\xi_m^{i-1}=\sum_{i=1}^{l(\mu)}\xi_m^{i-1}$, which leads to $m|l(\la)-l(\mu)$
using the fact that $\sum_{i=1}^r\xi_m^i=0$ if and only if $m|r$. Comparing the coefficients of $q$ in $f_\la(q)$ and $f_\mu(q)$, we should have $m\mid m_1(\la)-m_1(\mu)$.  But $m_1(\la),m_1(\mu)\in[0,m-1]$ thus $m_1(\la)=m_1(\mu)$. Similarly, we can show that $m_i(\la)=m_i(\mu)$ for $i=2,3,\cdots$. Therefore we have $\la=\mu$.

\begin{remark}\label{R:clacmu}
Using the same argument, we can show that $c_{\la\la}=c_{\mu\mu}$ if and only if $m_i(\la)\equiv m_i(\mu)\mod{m}$ for all $i\geq1$. Thus if $c_{\la\la}=c_{\mu\mu}$,  $\la\neq\mu$ and $|\la|=|\mu|$ then both $\la$ and $\mu$ have a multiplicity greater than $m$.
\end{remark}
\begin{corollary} Let $\epsilon_n=\frac{q^n-1}{1-\xi_m^n}c^{-n}$.
Then for each partition $\la\in\mathcal P_m$ (with $m_i(\la)<m$), there is a unique symmetric function $Q_\la$ in the algebra $\Lambda^{(m)}(\ep)$ such that\\
(1) $Q_\la= \sum_{\la\leq\mu\in\mathcal P_m}C_{\la\mu}q_\mu$, where $C_{\lambda\mu} \in\mathbb Q(\epsilon)$ with $C_{\la\la}=1,$\\
(2) $Q_\la$ is an eigenvector of  $X_0$.\\
Moreover these $Q_\la$'s give rise to an orthogonal basis of $\Lambda^{(m)}(\ep)$ and $X_0.Q_\la=c_{\la\la}Q_\la$ with $c_{\la\la}$ given in Theorem \ref{T:target operator}.
\end{corollary}
\begin{proof}

To prove the first statement we need to show that there is a unique set
of constants $C_{\la\mu}, \la\leq\mu\in\mathcal
P_m$ 
with $C_{\la\la}=1$ such that $\sum_{\la\leq\mu\in\mathcal P_m}C_{\la\mu}q_\mu$ is an eigenvector of $X_0$,
i.e.,
\begin{align}\label{E:findingcoefficients}
\sum_{\la\leq\mu\in\mathcal P_m}c_{\la\la}C_{\la\mu}q_\mu=\sum_{\la\leq\mu\in\mathcal P_m}C_{\la\mu}X_0\cdot q_\mu.
\end{align}
Here we used the fact that the eigenvalue is $c_{\la\la}$.
Let $X_0\cdot q_\mu=\sum_{\la\leq\mu\in\mathcal P_m} c'_{\mu\nu}q_\nu$ (see Theorem \ref{T:target operator}),
and we compare the coefficients of $q_\nu$ for $\nu\geq\la$ in both sides of (\ref{E:findingcoefficients}). For $\nu=\la$, the coefficients in both sides are already equal, and $C_{\la\la}=1$ is fixed. For $\nu>\la$, we solve that
\begin{align*}
C_{\la\nu}=\frac{\sum_{\nu>\mu\geq\la}C_{\la\mu}c'_{\mu\nu}}{c_{\la\la}-c_{\nu\nu}}.
\end{align*}
By induction on the dominance order we see that each $C_{\la\nu}$ is uniquely determined by this formula. This finishes the proof of the first statement.

 As for the second statement, we first have $X_0\cdot Q_\la=c_{\la\la}Q_\la$ by Theorem \ref{T:target operator} (and Remark \ref{r:eigenvalue}).
 It follows from the self-adjointness of $X_0$ that
\begin{align*}
&c_{\la\la}\langle Q_\la,Q_\mu\rangle=\langle X_0\cdot Q_\la,Q_\mu\rangle=\langle Q_\la,X_0\cdot Q_\mu\rangle\\
&=c_{\mu\mu}\langle Q_\la,Q_\mu\rangle.
\end{align*}
Therefore $Q_{\lambda}$ are orthogonal by the fact that $c_{\la\la}\neq c_{\mu\mu}$ for $\la\neq\mu$.
\end{proof}

\begin{remark} Set $c=\xi_m^{-1}$; then $\epsilon_n=\frac{1-q^n}{1-\xi_m^{-n}}$ corresponds to the specialization of $t=\xi_m^{-1}$ in  Macdonald functions. The vertex operator and  the eigenvalue are given by
\begin{align*}
&X(z)=\exp\Big(\sum_{m\nmid n\geq1}
\frac{z^{n}h_{-n}}{n} (\xi_m^{-n}-1)\Big) \exp\Big(\sum_{m\nmid n\geq1}
\frac{z^{-n}h_n}{n}(1-\xi_m^n) \Big),\\
&c_{\la\la}=1+(1-\xi_m)\sum_{i=1}^{l(\la)}(q^{\la_i}-1)\xi_m^{i-1}.
\end{align*}

Furthermore, if we consider $m=2$ ($\xi_m=-1$) and $q=0$, then $\epsilon_n=\frac12$ (for odd $n$), then $Q_\la$'s (for strict partition $\la$) are Schur Q-functions.
\end{remark}

\begin{remark} The symmetric functions we constructed bear certain similarity
with the deformation of Schur's Q-functions \cite{J2}
constructed via vertex operators similar to \cite{J1}. These $q$-Schur Q-functions 
were later systematically studied in \cite{TZ} where they
conjectured the positivity of the $q$-Kostka polynomials. In \cite{WW1, WW2}
a different deformation of Schur Q-functions was studied
from the viewpoint of representation theory of spin
groups.
\end{remark}

\centerline{\bf Acknowledgments}
We would like to thank Jinkui Wan and Weiqiang Wang for stimulating discussion
on a related problem.
The first and third authors thank China Scholarship Council for partial support and
MIT and NCSU respectively for hospitality during part of this work.
 The second author is grateful to the support of Simons Foundation grant 198129,
NSFC 11271138, Humboldt foundation 
and Max-Planck Institute for Mathematics in the Sciences
at Leipzig for hospitality 
during the work.

\bibliographystyle{amsalpha}

\end{document}